\documentclass[final]{article}

\usepackage{showkeys}
\usepackage{amsmath,amsthm,amssymb,amsfonts}
\usepackage{graphicx,color}

\newcommand{\R}{\mathbb{R}}
\renewcommand{\S}{\mathbb{S}}
\newcommand{\N}{\mathbb{N}}

\newcommand{\eps}{\varepsilon}

\DeclareMathOperator*{\osc}{osc}
\DeclareMathOperator{\tr}{tr}

\newtheorem{proposition}{Proposition}
\newtheorem{lemma}{Lemma}

\newtheorem{theorem}{Theorem}

\theoremstyle{definition}

\theoremstyle{remark}

\newtheorem{example}{Example}

\author{C. Imbert\footnote{CNRS, UMR 8050 \& Laboratoire d'analyse et
    de math\'ematiques appliqu\'ees, Universit\'e Paris-Est
    Cr\'eteil Val de Marne, 61 avenue du g\'en\'eral de Gaulle, 94010
    Cr\'eteil cedex, France} and L. Silvestre\footnote{Mathematics
    Department, University of Chicago, Chicago, Illinois 60637, USA}}

\title{$C^{1,\alpha}$ regularity of solutions of degenerate fully
  non-linear elliptic equations}

\begin{document}

\maketitle

\begin{abstract}  In the present paper, a class of fully
non-linear elliptic equations are considered, which are degenerate as the gradient becomes small. H\"older estimates obtained by the first author (2011) are combined with new Lipschitz estimates obtained through the Ishii-Lions method in order to get $C^{1,\alpha}$ estimates for solutions of these  equations.
\end{abstract}



\section{Introduction}

This paper is concerned with the study of the regularity of solutions
of the following non-linear elliptic equation
\begin{equation}\label{eq:main}
|\nabla u |^{\gamma} F(D^2 u) = f \qquad \text{ in } B_1
\end{equation}
where $B_1$ is the unit ball of $\R^d$ and $\gamma > 0$, $F$ is
  uniformly elliptic, $F(0)=0$ and $f$ is bounded. 

\paragraph{Singular/degenerate fully non-linear elliptic equations.}
Equation~\eqref{eq:main} makes part of a class of non-linear elliptic
equations studied in a series of papers by Birindelli and Demengel,
starting with \cite{bd04}. The specificity of these equations is that
they are not uniformly elliptic; they are either singular or
degenerate (in a way to be made precise). 

Birindelli and Demengel proved many important results in the singular
case such as comparison principles and Liouville type results
\cite{bd04}, regularity and uniqueness of the first eigenfunction
\cite{bd10} \textit{etc.} In the degenerate case, the set of results
\cite{bd07,bd09} is less complete and in particular, there was no
$C^{1,\alpha}$ estimate in the non-radial  case (see 
\cite{bd10b} for the radial case).

Alexandrov-Bakelman-Pucci (ABP) estimate were obtained for such
equations independently in \cite{dfq09} and \cite{i11}. It was used
to derive Harnack inequality in the singular case in \cite{dfq10} and in
both cases in \cite{i11}. From Harnack inequality, it is classical to
derive H\"older estimate (\cite{dfq10} in the singular case,
\cite{i11} in both cases). 

\paragraph{Main result.} The main result of this paper is the following
\begin{theorem}\label{thm:main}
Assume that $\gamma \geq 0$, $F$ is uniformly elliptic, $F(0)=0$,
  and $f$ is bounded in $B_1$.  There exists $\alpha>0$ and $C>0$ only
  depending on $\gamma$, the ellipticity constants of $F$ and
  dimension $d$, such that any viscosity solution $u$ of
  \eqref{eq:main} is $C^{1,\alpha}$ and
$$
[u]_{1+\alpha,B_{1/2}} \le C \left(||u||_{L^\infty} +
||f||_{L^\infty}^{ \frac1{1+\gamma}} \right). 
$$
\end{theorem}

\paragraph{Comments.} 
Getting $C^{1,\alpha}$ estimates consists in proving that the graph of
the function $u$ can be approximated by planes with an error bounded
by $C r^{1+\alpha}$ in balls of radius $r$. The proof is based on an
iterative argument, in which we show that the graph of $u$ gets
\emph{flatter} (meaning better approximated by planes) in smaller
balls. The iterative step, after a rescaling, amounts to show that if
$p\cdot x+u$ satisfies \eqref{eq:main} in $B_1$ with $\osc u \le 1$,
then the oscillation of $u$, up to a linear function $p'\cdot x$, is
smaller in a smaller ball. This is proved by compactness. In order to
make such an argument work, the modulus of continuity of $u$ has to be
controlled independently of the slopes $p$ and $p'$ which can vary
from one scale to the other. There is a difficulty since $u-p \cdot x$
does not satisfy any PDE independently of $p$. The main originality of
this paper is to combine the method introduced by Ishii and Lions
\cite{il90} to get Lipschitz estimate in the case of large slopes and
the Harnack inequality approach of Krylov-Safonov-Caffarelli
\cite{cc95} adapted in \cite{i11} to the present framework for small
slopes.

An alternative approach to find a modulus of continuity for solutions of the rescaled equation (see \eqref{eq:rescaled} below) for large slopes
could be to apply the Harnack inequality from \cite{savin2007small} to
get a uniform H\"older modulus of continuity for $|p|$ large enough
instead of the Ishii-Lions method to get a uniform Lipschitz
estimate. We chose the latter approach because of its simplicity.

The following example shows that solutions $u$ of \eqref{eq:main}
cannot be more regular than $C^{1,\alpha}$, even if $f$ is H\"older
continuous. 
\begin{example}
The function $u (x) = |x|^{1+\alpha}$ satisfies 
$$
|Du|^\gamma \Delta u
= C |x|^{(1+\alpha)(\gamma+1)-(\gamma+2)}
$$
 where
$C=(1+\alpha)^{1+\gamma} (d+\alpha -1)$. In particular, if we choose $\alpha = 1/(1+\gamma)$ the right hand side is simply constant. This example shows that even for a constant right hand side and $F(D^2 u)=\Delta u$, we cannot expect in general the solution to be more regular than $C^{1,\alpha}$ with $\alpha<1$. 
\end{example}

As far as the authors know, the result of Theorem \ref{thm:main} is new even for the simple equation $|\nabla u|^\gamma \Delta u = f(x)$. For this case we expect the optimal $\alpha$ to be in fact equal to $1/(1+\gamma)$ although we did not work on that issue. For general fully nonlinear equations $F(D^2 u)$ the value of $\alpha$ can get arbitrarily small even in the case $\gamma=0$ (see \cite{NV2} for an example).

The paper is organized as follows. In section~2 we specify the
notation to be used in the paper and we review a few well known
definitions and results for fully nonlinear elliptic equations. In
section~3 we restate Theorem 1 in a simplified form simply by
rescaling. In section~3, we also show how the iteration of the
\emph{improvement of flatness} lemma implies the main theorem. The
methods of section~3 are more or less standard for proving
$C^{1,\alpha}$ regularity for elliptic equations. In section~4 we find
a uniform modulus of continuity for the difference between the
solution and a plane appropriately rescaled. Based on this continuity
estimates we prove the improvement of oscillation lemma by a
compactness argument. In the last section we show a technical lemma
that says that viscosity solutions to $|\nabla u|^\gamma F(D^2 u)=0$
are also viscosity solutions to $F(D^2 u)=0$. This lemma is used to
characterize the limits in the compactness argument for the proof of
the improvement of flatness lemma in section~4.
\section{Preliminaries}

\subsection{Notation} 

For $r >0$, $B_r(x)$ denotes the open ball of
radius $r$ centered at $x$. $B_r$ denotes $B_r(0)$. $\S_d$ denotes the
set of symmetric $d \times d$ real matrices. $I$ denotes the identity
matrix. 

For $\alpha \in (0,1]$ and $Q \subset \R^d$, we consider
\begin{eqnarray*}
[u]_{\alpha,Q} &=& \sup_{x,y \in Q, x\neq y} \frac{u(x)-u(y)}{|x-y|^\alpha}, \\ 
{}[u]_{1+\alpha,Q} &=& \sup_{\rho>0, x \in Q}
\inf_{p \in \R^d} \sup_{z \in B_\rho(x) \cap Q} |u (z) - p\cdot z|.
\end{eqnarray*}

\subsection{Uniform ellipticity}

We recall the definition of uniform ellipticity (see \cite{cc95} for
more details). We say that a function $F$ defined on the set of real
symmetric matrices and taking real values is uniformly elliptic if
there exist two positive constants $\lambda$ and $\Lambda$ such that
for any two symmetric matrices $X$ and $Y$, with $Y \geq 0$ we have
\[ \lambda \tr Y \leq F(X) - F(X+Y) \leq \Lambda \tr Y.\]
The constants $\lambda$ and $\Lambda$ are called the \emph{ellipticity
  constants}. Under this definition $F(X) = -\tr (X)$ is uniformly
elliptic with ellipticity constants $\lambda=\Lambda=1$, and $F(D^2u)
= -\Delta u = f(x)$ is a uniformly elliptic equation.

The maximum and minimum of all the uniformly elliptic functions $F$
such that $F(0)=0$, are called the \emph{Pucci operators}. We write
them $P^+$ and $P^-$. Recall that $P^-$ has the closed form
\[ P^-(X) = -\Lambda \tr X^+ - \lambda \tr X^-, \]
where $\tr X^+$ is the sum of all positive eigenvalues of $X$ and $\tr
X^-$ is the sum of all negative eigenvalues of $X$. With the
definition of $P^+$ and $P^-$ at hand, it is equivalent that $F$ is
uniformly elliptic with the inequality
\[ P^-(Y) \leq F(X+Y) - F(X) \leq P^+(Y),\]
for any two symmetric matrices $X$ and $Y$.

\subsection{Two observations}

The uniform ellipticity hypothesis on $F$ implies that there exist
$\alpha_0 \in (0,1)$ and $C>0$ such that viscosity solutions of
$F(D^2u)=0$ in $B_1$ are $C^{1,\alpha_0}$ in the interior of $B_1$ and
$$
[u]_{1+\alpha_0,B_{1/2}} \le C ||u||_{L^\infty(B_1)}.
$$ 
The constants $\alpha_0$ and $C$ depend on the ellipticity constants
and dimension only.

Note that for any constant $a >0$, the function $a^{-1} F(aX)$ has
the same ellipticity constants as $F$. This will be important when
rescaling the equation.

\section{Reduction of the problem}

In this section, we first show that a simple rescaling reduces the
proof of the problem to the case that $||u||_{L^\infty}\leq 1/2$ and
$||f||_{L^\infty} \leq \eps_0$ for some small constant $\eps_0$ which
will be chosen later. We then further reduce the proof to an
\emph{improvement of flatness} lemma.

\subsection{Rescaling}

  We work with the arbitrary normalization $||u||_{L^\infty}\leq
1/2$ because that implies that $\osc u \leq 1$ and that will be a good
starting point for our iterative proof of $C^{1,\alpha}$ regularity.

\begin{proposition}
In order to prove Theorem~\ref{thm:main}, it is enough to prove
that \[[u]_{1+\alpha,B_{1/2}} \leq C\] assuming
$||u||_{L^\infty(B_1)} \leq 1/2$ and $||f||_{L^\infty(B_1)} \leq
\eps_0$ for some $\eps_0>0$ which only depends on the ellipticity
constants, dimension and $\gamma$. 
\end{proposition}

\begin{proof}
Given any function $u$ under the assumptions of Theorem
\ref{thm:main}, we can take $\kappa = \left(2||u||_{L^\infty} +
(||f||_{L^\infty}/\eps_0)^{1/(1+\gamma)} \right)^{-1}$ and consider
the scaled function $\tilde u(x) = \kappa u(x)$ solving the equation
\[ 
|\nabla \tilde u|^\gamma \kappa F(\kappa^{-1} D^2\tilde u) =
\kappa^{1+\gamma} f(x). 
\]
We previously made the observation that the function $\kappa
F(\kappa^{-1}X)$ has the same ellipticity constants as $F(X)$. But now
$||\tilde u||_{L^\infty} \leq 1/2$ and $||\tilde f||_{L^\infty} \leq
\eps_0$. Therefore, if 
\[ 
[\tilde u]_{1+\alpha,B_{1/2}} \leq C,
\]
by scaling back to $u$, we get
\[
[u]_{1+\alpha,B_{1/2}} \leq C \kappa^{-1} \le C (||u||_{L^\infty(B_1)}
+ ||f||_{L^\infty}^{1/(1+\gamma)})
\]
which concludes the proof.
\end{proof}

It is enough to prove that the solution $u$ of \eqref{eq:main} is
$C^{1,\alpha}$ at $0$ that is to say that there exists $C>0$ and
$\alpha$ (only depending on the ellipticity constants, dimension and
$\gamma$) such that for all $r \in (0,1)$, there exists $p \in \R^d$
such that
\begin{equation} \label{eq:pointC1alpha}
\osc_{B_r} (u - p \cdot x) \le C r^{1+\alpha}. 
\end{equation}

If we start with a function $u$ such that $\osc_{B_1} u \le 1$, we
already have the inequality for $r=1$ with $C=1$. In order to get such
a result for all $r \in (0,1)$, it is enough to find $\rho,\alpha \in (0,1)$
such that for all $k \in \N$ there exists $p_k \in \R^d$ such that
$$
\osc_{B_{\rho^k}} (u -p_k \cdot x) \le \rho^{k(1+\alpha)}.
$$
The inequality \eqref{eq:pointC1alpha} follows with $C = \rho^{-(1+\alpha)}$.

This is the reason why we consider $r_k = \rho^k$ and we aim at
proving by induction on $k \in \N$ the following
\begin{lemma}\label{lem:induction}
There exists $\rho, \alpha \in (0,1)$ and $\eps_0 \in [0,1]$ only
depending on $\gamma$, ellipticity constants and dimension such that,
as soon as a viscosity solution $u$ of \eqref{eq:main} with
$||f||_{L^\infty} \leq \eps_0$ satisfies $\osc_{B_1} u \leq 1$, then
for all $k \in \N$, there exists $p_k \in \R^d$ such that
\begin{equation}\label{eq:induction}
\osc_{B_{r_k}} (u - p_k \cdot x) \le r_k^{1+\alpha}. 
\end{equation}
\end{lemma}

The choice of $\rho$ depends on the $C^{1,\alpha_0}$ estimates for
$F(D^2u)=0$. Precisely, since we assume that any viscosity solution $u$ of
$F(D^2u )=0$ in $B_1$ is $C^{1,\alpha_0}$, it is in particular
$C^{1,\alpha_0}$ at $0$, that is to say there exists $C_0>0$ such
that for all $r \in (0,1)$, there exists $p \in \R^d$ such that
$$
\osc_{B_r} (u- p \cdot x) \le C_0 r^{1+\alpha_0}.
$$ 
We then pick $\rho \in (0,2^{-\gamma-1})$ such that 
\begin{equation}\label{eq:rho}
C_0 \rho^{\alpha_0} \le \frac14.
\end{equation}
Given a solution $u$ of $F=0$ in $B_1$, we also pick $p_\rho=p_\rho (u)$ such that
\begin{equation}\label{eq:p-rho}
\osc_{B_\rho} (u- p_\rho \cdot x) \le \frac14 \rho.
\end{equation}

\subsection{Reduction to the \emph{improvement of flatness} lemma}

In order to prove Lemma~\ref{lem:induction}, we prove an
\emph{improvement of flatness} lemma; it is the core of the paper. It
basically says that if $p\cdot x + u$ solves \eqref{eq:main} in $B_1$
and the oscillation of $u$ in $B_1$ is less than $1$, say, then the
function $u$ can be approximated by a linear function in a smaller
ball with an error that is less than the radius of the ball. We make
this statement rigourous and quantitative now.
\begin{lemma}[Improvement of flatness lemma]\label{lem:core}
There exists $\eps_0 \in [0,1]$ and $\rho \in (0,1)$ only depending on
$\gamma$, ellipticity constants and dimension such that, for any $p
\in \R^d$ and any  viscosity solution $u$ of
\begin{equation}\label{eq:rescaled}
|p+ \nabla u|^\gamma F(D^2u ) = f \qquad \text{ in } B_1
\end{equation}
such that $\osc_{B_1} u \le 1$ and $\|f\|_{L^\infty(B_1)} \le
\eps_0$,  there exists $p' \in \R^d$ such that 
$$ \osc_{B_\rho} ( u - p'\cdot x) \le \frac12 \rho. $$
\end{lemma}

It is important to remark that the choice of $\rho$ and $\eps_0$ works
for all vectors $p$ in the previous Lemma. No constant depends on $p$.

We now explain how to derive Lemma~\ref{lem:induction} from
Lemma~\ref{lem:core}.
\begin{proof}[Proof of Lemma~\ref{lem:induction}.]
For $k=0$, we simply choose $p_0=0$ and \eqref{eq:induction} is
guaranteed by the assumption $\osc u \leq 1$.

We choose $\alpha>0$ small such that $\rho^\alpha > 1/2$.

We assume now that $k \ge 0$ and that we constructed already $p_k \in
\R^d$ such that \eqref{eq:induction} holds true. We then consider for
$x \in B_1$,
$$
u_k (x) = r_k^{-1-\alpha} [u(r_kx) - p_k \cdot (r_k x)].
$$
The vector $p_k$ is such that $\osc_{B_1} u_k \le 1$. Moreover, $u_k$ satisfies 
$$
|r_k^{-\alpha} p_k + D u_k|^\gamma \, r^{1-\alpha} F(r^{\alpha-1} D^2 u_k) = f_k(x)$$
 with $f_k (x) = r_k^{1-\alpha(1+\gamma)} f(r_k x)$. In particular, $\|f_k\|_{L^\infty
  (B_1)} \le \eps_0$ as long as $\alpha < 1/(1+\gamma)$.
  
Notice that the function $r^{1-\alpha} F(r^{\alpha-1} X)$ has the same
ellipticity constants as $F(X)$, therefore the $C^{1,\alpha_0}$
estimates are conserved by this scaling.
  
Now we apply Lemma~\ref{lem:core} and get $q_{k+1}$ such that 
$$
\osc_{B_\rho} (u_k - q_{k+1}\cdot x) \le \frac12 \rho
$$
Because of our choice of $\alpha$, we then obtain $p_{k+1}$ such that
$$
\osc_{B_{r_{k+1}}} (u - p_{k+1}\cdot x) \le r_k^{1+\alpha} \frac 12
\rho \le  r_{k+1}^{1+\alpha}.
$$ 
The proof is now complete.
\end{proof}

\section{Equi-continuity of rescaled solutions}

The proof of Lemma~\ref{lem:core} relies on the following lemma in
which the modulus of continuity of solutions of \eqref{eq:rescaled} is
controlled.
\begin{lemma}[Modulus of continuity independent of $p$]\label{lem:equicont}
For all $r>0$, there exist $\beta \in (0,1)$ and $C >0$ only
depending on ellipticity constants, dimension, $\gamma$ and $r$ and such
that for all viscosity solution $u$ of \eqref{eq:rescaled} with $\osc_{B_1}
u \leq 1$ and $||f||_{L^\infty(B_1)} \leq \eps_0 < 1$ satisfies
\begin{equation}\label{eq:equicont}
[u]_{\beta,B_r} \le C.
\end{equation}
In particular, the modulus of continuity of $u$ is controlled
independently of $p$. 
\end{lemma}

\subsection{Proof of Lemma~\ref{lem:equicont}}

This lemma is a consequence of the two following ones.
\begin{lemma}[Lipschitz estimate for large $p$'s]\label{lem:equilip}
Assume $u$ solves \eqref{eq:rescaled} with $\osc_{B_1} u \leq 1$ and
$||f||_{L^\infty(B_1)} \leq \eps_0 < 1$. If $|p| \ge 1/a_0$, with
$a_0=a_0(\lambda,\Lambda, d, \gamma, r)$, 
then any viscosity solution $u$ of \eqref{eq:rescaled} is Lipschitz continuous in $B_r$ and
\begin{equation}\label{eq:estim-lip}
[u]_{1,B_r} \le C 
\end{equation}
where $C=C(\lambda,\Lambda,\gamma,d,r)$.
\end{lemma}
\begin{lemma}[H\"older estimate for small $p$'s]\label{lem:equiholder}
Assume $u$ solves \eqref{eq:rescaled} with $\osc_{B_1} u \leq 1$ and
$||f||_{L^\infty(B_1)} \leq \eps_0 < 1$. If $|p| \le 1/a_0$, then $u$ is
$\beta$-H\"older continuous in $B_r$ and 
$$
[u]_{\beta,B_r} \le C
$$ 
where $\beta=\beta(\lambda,\Lambda,d,r,a_0)$ and $C=C(\lambda,\Lambda,d,r,a_0)$.
\end{lemma}
We now turn to the proof of these two lemmas.
\begin{proof}[Proof of Lemma~\ref{lem:equilip}]
We rewrite \eqref{eq:rescaled} as 
$$
|e + a Du|^\gamma F(D^2 u ) = \tilde f
$$ 
where $e = p / |p|$ and $a= 1 / |p| \in [0,a_0]$ and 
$$
\tilde f= |p|^{-\gamma}f.
$$ 
Remark that 
$$
\| \tilde f \|_{L^\infty(B_1)} \le a_0^\gamma \eps_0. 
$$ We use viscosity solution techniques first introduced in
\cite{il90}. For all $x_0 \in B_{r/2}$, we look for $L_1>0$ and
$L_2>0$ such that
$$ 
M = \sup_{x,y \in B_r} u(x)-u(y) -L_1 \omega (|x-y|) - L_2
|x-x_0|^2 - L_2 |y-x_0|^2 \le 0
$$ 
where $\omega (s) = s - \omega_0 s^{\frac32}$ if $s \le
s_0:=(2/3\omega_0)^2$ and $\omega(s)=\omega(s_0)$ if $s \ge s_0$. We
 choose $\omega_0$ such that $s_0 \ge 1$. We notice that if we
proved such an inequality, the Lipschitz constant is bounded from
above by any $L>L_1$.

We argue by contradiction by assuming that $M>0$. If $(x,y) \in
\bar{B}_r \times \bar{B}_r $ denotes a point where the maximum is
reached (recall that $u$ is continuous and its oscillation is
bounded), we conclude that
$$
L_1 \omega (|x-y|) + L_2 |x-x_0|^2 +L_2 |y-x_0|^2\le \osc_{B_1} u \le 1. 
$$ 

We choose $L_2=(4/r)^2$, so that $|x-x_0| \le \frac{r}4$ and $|y-x_0|\le
\frac{r}4$. With this choice, we force the points $x$ and $y$ where the
supremum is achieved to be in $B_{r}$. Remark also that the supremum
cannot be reached at $(x,y)$ with $x=y$, otherwise $M \le 0$. Hence,
we can write two viscosity inequalities.

Before doing so, we compute the gradient of the test-function for $u$
with respect to $x$ and $y$ at $(x,y)$
$$
q_x = q+ 2 L_2 x \qquad \text{and}\qquad q_y = q - 2 L_2 y
$$ where $q=L_1 \omega'(|\delta|) \hat \delta$, $\delta = x-y$ and
$\hat \delta = \delta / |\delta|$.  To get appropriate viscosity
inequalities, we shall use Jensen-Ishii's Lemma  in order to
construct a limiting sub-jet $(q_x,X)$ of $u$ at $x$ and a limiting
super-jet $(q_y,Y)$ of $u$ at $y$ such that the following $2n \times
2n$ matrix inequality holds for all $\iota >0$ small enough (depending
on the norm of $Z$):
\[
\begin{pmatrix}
X & 0 \\
0 & -Y 
\end{pmatrix}
\leq 
\begin{pmatrix}
Z  & -Z \\
-Z & Z
\end{pmatrix} + (2L_2+\iota) I
\]
where $Z=L_1 D^2 (\omega(|\cdot|)) (x-y)$.  We refer the reader to
\cite{bci11,bcci11} for details. Applying the previous matrix
inequality as a quadratic form inequality to vectors of the form
$(v,v)$ we obtain
\begin{equation} \label{eq:noneVeryPositive}
 \langle (X-Y) v,v \rangle \leq (4 L_2+\iota) |v|^2.
\end{equation}
Therefore $X-Y \leq (4L_2+\iota) I$, or equivalently, all eigenvalues of $X-Y$
are less than $4L_2+\iota$. On the other hand, applying now the particular
vector $(\hat \delta, -\hat \delta)$, we obtain
\begin{equation} \label{eq:oneVeryNegative}
 \langle (X-Y) \hat \delta, \hat \delta \rangle \leq (4 L_2 +\iota -
 6\omega_0 L_1 |x-y|^{-1/2} ) |\hat \delta|^2 \leq (4 L_2 +\iota- 3\sqrt 2
 \omega_0 L_1 ) |\hat \delta|^2.
\end{equation}
Thus, at least one eigenvalue of $X-Y$ is less than $(4 L_2 +\iota-
3\omega_0 \sqrt{2} L_1 )$ (which will be a negative number).  We next
consider the minimal Pucci operator $P^-$. We recall that $-P^-(A)$
equals $\lambda$ times the sum of all negative eigenvalues of $A$
plus $\Lambda$ times the sum of all positive eigenvalues. Therefore,
from \eqref{eq:noneVeryPositive} and \eqref{eq:oneVeryNegative}, we
obtain
\begin{align*}
 P^-(X-Y) &\ge -\lambda (4L_2 +\iota- 3 \sqrt 2 \omega_0 L_1) - \Lambda (d-1) (4 L_2 +\iota)\\
 &\ge -(\lambda+(d-1)\Lambda)(4 L_2+\iota) + 3 \sqrt 2 \omega_0 \lambda L_1.
\end{align*}

We now write the two viscosity inequalities and we combine them in
order to get a contradiction.
\begin{align*}
|e+a q_x|^\gamma F(X) & \le \tilde{f}(x) \\ 
|e+a q_y|^\gamma F(Y) & \ge \tilde{f}(y).
\end{align*}
We will choose $a_0$ small enough depending on $L_1$ and $L_2$ so that
$|a q_x| \le 1/2$ and $|a q_y| \le 1/2$. The constant $L_1$ will be
chosen later and its value does not depend on this choice of $a_0$. In
particular, we have
$$
\frac12 \le \min (|e+aq_x|, |e+aq_y|).
$$
We now use that $F$ is uniformly elliptic to write
$$
F(X)  \ge F(Y) + P^-(X-Y).  
$$ Combining the previous displayed inequalities and recalling
$||f||_{L^\infty} \leq \eps_0$ yields
$$ 
3 \sqrt{2} \omega_0 \lambda L_1 \le  (\lambda + \Lambda (d-1)) (4L_2+\iota)
+ 2^{\gamma+1} \eps_0.
$$
Choosing $L_1$ large enough depending on $\lambda$, $\Lambda$, $d$,
$\gamma$, and the previous choice of $L_1$ (which depends on $r$
only), we obtain a contradiction.

Note that this choice of $L_1$ does not depend on the previous choice
of $a_0$, so we should first choose $L_1$ large and then $a_0$ small.
The proof of the lemma is now complete.
\end{proof}

\begin{proof}[Proof of Lemma~\ref{lem:equiholder}]
The equation can be written as $G(Du,D^2u)=f$ with 
$$
G(q,X)=|p+q|^\gamma F(X).
$$
 In particular, if $|q| \ge 2 a_0^{-1}$ then
$|p+q|^\gamma \ge a_0^{-\gamma}$. In particular,
$$
\left.\begin{array}{r} G(q,X)=0 \\ |q|\ge 2a_0^{-1} \end{array}\right\}
\Rightarrow \left\{ \begin{array}{l}
 P^+ (D^2u)+ a_0^\gamma \frac{|f|}{A^\gamma} \ge 0 \\
 P^- (D^2u) - a_0^\gamma \frac{|f|}{A^\gamma} \le 0
\end{array}\right.
$$ where $P^\pm$ denote extremal Pucci's operators
associated with the ellipticity constants of $F$.  We know from
\cite{i11} that there exists $\beta_1 \in (0,1)$ and $C_1$ only
depending on $r$, dimension and ellipticity constants of $F$ such that
\begin{align*}
[u]_{\beta_1,B_r} & \le  C_1
\left(\osc_{B_1} u+ \max(2a_0^{-1},\|f\|_{L^n(B_1)})
\right) \\
& \le C_1 (1 + \max (2a_0^{-1}, \eps_0)). 
\end{align*} 
The proof of the lemma is now complete.
\end{proof}

\subsection{Proof of the \emph{improvement of flatness} Lemma}

With Lemma~\ref{lem:equicont} in hand, we can now turn to the proof of
Lemma~\ref{lem:core}.
\begin{proof}[Proof of Lemma~\ref{lem:core}]
We argue by contradiction and we assume that there exist sequences
$\eps_n \to 0$, $p_n \in \R^d$, $f_n$ such that
$\|f_n\|_{L^\infty (B_1)} \le \eps_n$, and $u_n$ satisfying
\eqref{eq:rescaled} with $(p,f)=(p_n,f_n)$ such that for all $p'
\in \R^d$,
$$
\osc_{B_\rho} (u_n - p' \cdot x) > \frac12 \rho.
$$
Remark that $f_n \to 0$ as $n \to \infty$. 

Thanks to Lemma~\ref{lem:equicont}, we can extract a subsequence of
$(u_n)_n$ converging locally uniformly in $B_1$ to a continuous function
$u_\infty$. Remark that we have in particular for all $p' \in \R^d$, 
\begin{equation}\label{osc:lim}
\osc_{B_\rho} (u_\infty - p' \cdot x) > \frac12 \rho.
\end{equation}
We are going to prove that $u_\infty$ satisfies $F(D^2 u_\infty)=0$ in
$B_1$. This will imply that there exists a vector $p_\rho$ such that
\eqref{eq:p-rho} holds true. This is the desired contradiction with
\eqref{osc:lim}.

To prove that $F(D^2 u_\infty)=0$ in $B_1$, we now distinguish two cases.

If we can extract a converging subsequence of $p_n$, then we also do
it for $u_n$ and we get at the limit
$$
|p_\infty + \nabla u_\infty |^\gamma F(D^2 u_\infty) = 0
\qquad \text{ in } B_1. 
$$ 
In particular, we have $F(D^2 u_\infty)=0$ in $B_1$ (see
Lemma~\ref{lem:visc} in the next subsection). 

If now we cannot extract a converging subsequence of $p_n$, then
$|p_n| \to \infty$ and in this case, we extract a converging
subsequence from $e_n = p_n/|p_n|$ and dividing the equation by $|p_n|$ we get at the limit
$$
|e_\infty + 0 \nabla u_\infty|^\gamma F(D^2 u_\infty)=0 \qquad \text{ in }
B_1
$$
for $e_\infty \neq 0$ so that we also have in this case $F(D^2
u_\infty)=0$ in $B_1$. The proof of the lemma is now complete.
\end{proof}

\section{Viscosity solutions of $\mathbf{|\nabla u|^{\gamma}F(D^2 u)=0}$}

In the previous subsection, we used the following lemma. 
\begin{lemma}\label{lem:visc}
Assume that $u$ is a viscosity solution of 
$$
|p + \nabla u |^\gamma F(D^2 u)= 0 \qquad \text{ in } B_1.
$$  Then $u$ is a viscosity solution of $F(D^2
u)=0$ in $B_1$.
\end{lemma}
\begin{proof}
We reduce the problem to $p=0$ as follows. The function $v= u + p \cdot x$ satisfies 
$|\nabla v|^\gamma F(D^2 v)=0$ in $B_1$. If we proved the result for
$p=0$, we conclude that $F(D^2u) = F(D^2v) = 0$ in $B_1$. 

We now assume that $p=0$. We only prove the super-solution
property since the sub-solution property is very similar. 

Consider a test-function $\phi$ touching $u$ strictly from below at $x
\in B_1$. We assume for simplicity that $x=0$. Hence, we have,
$\phi(0)=u(0)=0$ and $\phi < u$ in $B_r \setminus \{0\}$ for some
$r>0$. We can assume without loss of generality that $\phi$ is
quadratic: $\phi (x) = \frac12 Ax\cdot x + b\cdot x$. If $b \neq 0$,
then we get the desired inequality: $F(A) \ge 0$.

If $b=0$, we  argue by contradiction by
assuming that $F(A) < 0$. Since $F$ is uniformly elliptic, this
implies that $A$ has a least one positive eigenvalue. Let $S$ be the
direct sum of eigensubspace corresponding to non-negative
eigenvalues. Let $P_S$ denote the orthogonal projection on $S$. We
then consider the following test function
$$
\psi (x) = \phi(x) + \eps |P_S x|. 
$$ Since $\phi<u$ in $B_r$, then $u - \psi$ reaches its minimum at $x_0$ in
$\bar{B}_r$ in the interior of the ball for $\eps$ small enough.

We claim first that $P_S x_0 \neq 0$. Indeed, if this is not true, we use the fact that 
$$
|P_s x| = \min_{|e|=1} e \cdot P_S x 
$$
and we deduce that for all $e \in \R^d$ such that $|e|=1$, the
test-function $\phi(x) + \eps e \cdot P_S x$ touches $u$ at $x_0$ and
we thus have for all such $e$'s
$$
|A x_0 + \eps P_S e|^\gamma F(A) \ge 0. 
$$
Hence, there exists such an $e$ such that $D \phi (x_0)+\eps P_S e
\neq 0$ and we get the contradiction $F(A) \ge 0$.

Since $P_S x_0 \neq 0$, $\psi$ is smooth in a neighbourhood of $x_0$
and we get the following viscosity inequality 
$$
 |A x_0 + \eps e_0|^\gamma F(A+\eps B) \ge 0
$$
where $e_0 = P_S x_0 /|P_S x_0|$ and $B \ge 0$ since $x \mapsto |P_S
x|$ is convex. Remark next that 
$$
(Ax_0+ \eps e_0) \cdot P_S x_0 = P_S A x_0 \cdot x_0 + \eps |P_S x_0|
\ge \eps |P_S x_0 | >0. 
$$
Hence $Ax_0 + \eps e_0 \neq 0$ and we get the following contradiction
$$
F(A) \ge F(A+\eps B) \ge 0.
$$
The proof is now complete.
\end{proof}

\paragraph{Acknowledgements.} The authors are grateful to
I.~Birindelli and F.~Demengel for pointing out this open problem and
for the fruitful discussions they had together.

Luis Silvestre was partially supported by the Sloan fellowship and NSF grants DMS-1065979 and DMS-1001629.

\nocite{*} \bibliographystyle{siam} \bibliography{degh}

\end{document}